\documentclass[12pt,reqno]{amsart}
\usepackage{amssymb,latexsym,graphicx,amscd,upgreek}
\usepackage{lscape}
\usepackage[all]{xy}
\usepackage{color}
\usepackage{epic,eepic}
\usepackage{xspace}
\usepackage{mathrsfs}
\usepackage{setspace}
\usepackage{cases}
\usepackage{bbm}
\usepackage{hyperref}
\hypersetup{
    colorlinks,
    linkcolor={blue},
    citecolor={blue},
    urlcolor={blue!80!black}
}

\newtheorem{Thm}{Theorem}[section]
\newtheorem{Lem}[Thm]{Lemma}
\newtheorem{Cor}[Thm]{Corollary}
\newtheorem{Prop}[Thm]{Proposition}

\newtheorem{Qu}[Thm]{Question}
\newtheorem{Qus}[Thm]{Questions}

\theoremstyle{definition}

\newcommand{\Example}{\noindent{\bf Example}}
\newcommand{\Examples}{\noindent{\bf Examples}}

\newcommand{\Remark}{\noindent{\bf Remark}}

\newcommand{\N}{\mathbb{N}}

\newcommand{\df}{\colon}

\newcommand{\cI}{{\mathcal I}}

\newcommand{\cO}{{\mathcal O}}
\newcommand{\cP}{{\mathcal P}}

\newcommand{\cU}{{\mathcal U}}
\newcommand{\cV}{{\mathcal V}}

\newcommand{\cZ}{{\mathcal Z}}

\newcommand{\bd}{\mathbf{d}}

\newcommand{\rk}{\operatorname{rk}}

\newcommand{\md}{\operatorname{mod}}

\newcommand{\ind}{\operatorname{ind}}

\newcommand{\Spec}{\operatorname{Spec}}

\newcommand{\gldim}{\operatorname{gl.dim}}
\newcommand{\pdim}{\operatorname{proj.dim}}

\newcommand{\dimv}{\underline{\dim}}

\newcommand{\soc}{\operatorname{soc}}

\newcommand{\tp}{\operatorname{top}}

\newcommand{\proj}{\operatorname{proj}}
\newcommand{\inj}{\operatorname{inj}}

\newcommand{\Hom}{\operatorname{Hom}}
\newcommand{\Ext}{\operatorname{Ext}}

\newcommand{\Aut}{\operatorname{Aut}}

\newcommand{\taureg}{\operatorname{\tau-reg}}
\newcommand{\rigid}{\operatorname{rigid}}
\newcommand{\partilt}{\operatorname{partilt}}

\newcommand{\taurigid}{\operatorname{{\tau-rigid}}}

\newcommand{\Ima}{\operatorname{Im}}

\newcommand{\Coker}{\operatorname{Cok}}

\newcommand{\ov}{\overline}

\newcommand{\irr}{\operatorname{Irr}}

\newcommand{\bsm}{\begin{smallmatrix}}
\newcommand{\esm}{\end{smallmatrix}}

\newcommand{\bbsm}{\left[\begin{smallmatrix}}
\newcommand{\besm}{\end{smallmatrix}\right]}

\newcommand{\bbm}{\begin{matrix}}
\newcommand{\ebm}{\end{matrix}}

\newcommand{\GL}{\operatorname{GL}}

\newcommand{\calU}{\mathcal{U}}
\newcommand{\calZ}{\mathcal{Z}}

\begin{document}

\date{14.05.2026}

\title{On the additivity of projective presentations of maximal rank}

\author{Grzegorz Bobi\'nski}
\address{Grzegorz Bobi\'nski\newline
Faculty of Mathematics and Computer Science\newline
Nicolaus Copernicus University\newline
ul. Chopina 12/18\newline
87-100 Toru\'n\newline
Poland}
\email{gregbob@mat.umk.pl}

\author{Jan Schr\"oer}
\address{Jan Schr\"oer\newline
Mathematisches Institut\newline
Universit\"at Bonn\newline
Endenicher Allee 60\newline
53115 Bonn\newline
Germany}
\email{schroer@math.uni-bonn.de}


\begin{abstract}
We study projective presentations of finite-dimensional modules over finite-dimensional algebras. We discuss if
projective presentations of maximal rank behave additively.
More precisely, we ask
if the direct sum of copies of a projective presentation of maximal rank is again of maximal rank.
The modules which have a projective presentation of maximal rank
are exactly the $\tau$-regular modules.
This class of modules can be seen as a generalization of
modules of projective dimension at most one, and of $\tau$-rigid modules.
The $\tau$-regular modules form
open subsets of module varieties.
Their closures are therefore unions of irreducible components, which are
called generically $\tau$-regular.
We discuss when a $\tau$-regular module or a generically $\tau$-regular component can be reduced to
a module or component of projective dimension at most one,
and we show that this is closely related to the question on
the additivity of maximal rank presentations.
\end{abstract}

\maketitle

\setcounter{tocdepth}{1}
\numberwithin{equation}{section}
\tableofcontents

\parskip2mm


\section{Introduction and main results}\label{sec:intro}


Throughout, $K$ denotes an algebraically closed field, and
$A$ is always a finite-dimensional associative $K$-algebra.
Let $\md(A)$ be the category of finite-dimensional left $A$-modules.
Let $\tau_A$ be the Auslander-Reiten translation for $A$.

\subsection{Projective presentations of maximal rank}
For $M,N \in \md(A)$ let
$$
r(M,N) := \max\{ \rk(f) \mid f \in \Hom_A(M,N) \},
$$
where $\rk(f)$ is the $K$-dimension of the image $\Ima(f)$ of $f$.

For $M \in \md(A)$ let
$$
P_1 \xrightarrow{f} P_0 \to M \to 0
$$
be a projective presentation of $M$.
We are mainly interested in the situation, where
$\rk(f) = r(P_1,P_0)$.
In this case,
we proved in \cite{BS25} that $M$ is $\tau$-regular, see
Section~\ref{subsec:introtauregular} for more details.

It is obvious that
$$
r(P_1^t,P_0^t) \ge t \cdot r(P_1,P_0)
$$
for all $t \ge 1$.

\begin{Qu}\label{qu:1}
When do we have
$$
r(P_1^t,P_0^t) = t \cdot r(P_1,P_0)
$$
for all $t \ge 1$?
\end{Qu}

\subsection{$\tau$-regular modules}\label{subsec:introtauregular}
This article deals with $\tau$-regular modules, which include all
$\tau$-rigid modules and also all modules of projective dimension at most one.

Let $S(1),\ldots,S(n)$ be the simple $A$-modules, up to isomorphism.
The \emph{dimension vector} of $M \in \md(A)$ is defined as
$$
\dimv(M) := ([M:S(1)],\ldots,[M:S(n)]) \in \N^n,
$$
where $[M:S(i)]$ denotes the \emph{Jordan-H\"older multiplicity} of
the simple module $S(i)$ in any composition series of $M$.
The module $M$ is \emph{sincere} if $[M:S(i)] \not= 0$ for
all $1 \le i \le n$.

For $\bd \in \N^n$ let
$\md(A,\bd)$ be the affine variety of $A$-modules
with dimension vector $\bd$.
For a definition of $\md(A,\bd)$ we refer to 
\cite[Section~3.1]{BS25}.

The corresponding product $\GL_\bd$ of general linear groups
acts on $\md(A,\bd)$ by conjugation.
The orbits of this action correspond to the isomorphism classes of $A$-modules with dimension vector $\bd$.
For $M \in \md(A,\bd)$ we denote its orbit by $\cO_M$.
Let $\irr(A,\bd)$ be the set of irreducible components of
$\md(A,\bd)$, and let
$$
\irr(A) := \bigcup_{\bd \in \N^n} \irr(A,\bd).
$$

For $M \in \md(A,\bd)$ let
$$
\md_M(A,\bd) :=
\bigcup_{\substack{\calZ \in \irr(A,\bd)\\M \in \calZ}} \calZ,
$$
and define
\begin{align*}
c_A(M) &:= \dim \md_M(A,\bd) - \dim \cO_M,
\\
e_A(M) &:= \dim \Ext_A^1(M,M),
\\
E_A(M) &:= \dim \Hom_A(M,\tau_A(M)).
\end{align*}

The number $E_A(M)$ is the
\emph{$E$-invariant} of $M$.
The module $M$ is called \emph{$\tau$-rigid} (or $\tau_A$-rigid)
if $E_A(M) = 0$.
These modules play a crucial role
in Derksen, Weyman and Zelevinsky's
\cite{DWZ08, DWZ10} representation theoretic approach to Fomin-Zelevinsky cluster algebras.
The study of $\tau$-rigid modules is part of $\tau$-tilting theory, which can be seen as a vast generalization of the classical tilting theory.
We refer to \cite{AIR14} for more details.

The following result is a direct consequence of Voigt's Lemma  \cite[Proposition~1.1]{G74} and the
Auslander-Reiten formulas.

\begin{Prop}\label{prop:voigt}
For all $M \in \md(A)$ we have
$$
c_A(M) \le e_A(M) \le E_A(M).
$$
\end{Prop}

Let $\Spec(A,\bd)$ be the affine scheme of $A$-modules
with dimension vector $\bd$.
For $M \in \md(A,\bd)$ the following are equivalent:
\begin{itemize}\itemsep2mm

\item[(i)]
$M$ is regular in $\Spec(A,\bd)$;

\item[(ii)]
$c_A(M) = e_A(M)$.

\end{itemize}
This follows again from Voigt's Lemma.

We call $M$ \emph{$\tau$-regular} (or \emph{$\tau_A$-regular})
if
$$
c_A(M) = E_A(M).
$$

\begin{Thm}[{\cite[Theorem~1.1]{BS25}}]\label{thm:BS25a}
Let $M \in \md(A)$, and
let
$$
P_1 \xrightarrow{f} P_0 \to M \to 0
$$ be a minimal projective presentation of $M$.
Then the following are equivalent:
\begin{itemize}\itemsep2mm

\item[(i)]
$M$ is $\tau$-regular;

\item[(ii)]
$\rk(f) = r(P_1,P_0)$.

\end{itemize}
\end{Thm}

Let
$$
\md^\tau(A,\bd) := \{ M \in \md(A,\bd) \mid \text{$M$ is
$\tau$-regular} \}.
$$
It is proved in \cite{BS25} that
$\md^\tau(A,\bd)$ is open in $\md(A,\bd)$.
As a consequence the Zariski closure
$$
\ov{\md^\tau(A,\bd)}
$$
is a union of irreducible components of $\md(A,\bd)$.
The set of these components is denoted by $\irr^\tau(A,\bd)$.
Let
$$
\irr^\tau(A) := \bigcup_{\bd \in \N^n} \irr^\tau(A,\bd).
$$
The components in $\irr^\tau(A)$ are called
\emph{generically $\tau$-regular} or just
\emph{$\tau$-regular} for short.
This class of irreducible components was (under a different name) first introduced in \cite{GLS12}.

For each $\tau$-regular $M \in \md(A)$ there is a unique component $\cZ \in \irr(A)$
with $M \in \cZ$.

An irreducible constructible subset $\cU$ of $\md(A,\bd)$ is
called \emph{generically $\tau$-regular} or just
\emph{$\tau$-regular} (resp. \emph{generically regular})
if it contains a $\tau$-regular (resp. regular) module.
In this case, the $\tau$-regular (resp. regular) modules form a dense open subset of $\cU$.

Throughout,
we use the word \emph{generic} in a slightly vague way.
For an irreducible constructible subset $\cU$ of some
affine variety,
we say that $M \in \cU$ is \emph{generic} in $\cU$ if it is contained
in some (small enough) dense open subset $\cV$ of $\cU$.
This subset $\cV$ consists of
the $M \in \cU$ such that some
upper (or lower) semicontinuous maps take their minimal (or maximal) value on $M$.
It will always be clear from the context, which semicontinuous maps
we are using.

\subsection{Hierarchy of modules}

The $\tau$-regular modules are a generalization of modules of projective dimension at most one, and of $\tau$-rigid modules.
Recall that $M \in \md(A)$ is \emph{rigid} if $\Ext_A^1(M,M) = 0$.
A rigid module $M$ is a \emph{partial tilting module} if
$\pdim(M) \le 1$.
Partial tilting modules are the direct summands of
tilting modules, which are the corner stone of classical tilting theory.

Let
\begin{align*}
\proj(A) &:= \{  M \in \md(A) \mid M \text{ is projective} \},
\\
\cP_{\le 1}(A) &:= \{ M \in \md(A) \mid \pdim(M) \le 1 \},
\\
\rigid(A) &:= \{ M \in \md(A) \mid \Ext_A^1(M,M) = 0 \},
\\
\partilt(A) &:= \{ M \in \md(A) \mid \Ext_A^1(M,M) = 0,\;
\pdim(M) \le 1 \},
\\
\taurigid(A) &:= \{ M \in \md(A) \mid \Hom_A(M,\tau_A(M)) = 0 \},
\\
\taureg(A) &:= \{ M \in \md(A) \mid M \text{ is $\tau$-regular} \}.
\end{align*}
We have the following hierarchy of classes of modules:
$$
\xymatrix@!@-6ex{
&\taureg(A) \ar@{-}[dl]\ar@{-}[dr] && \rigid(A)\ar@{-}[dl]
\\
\cP_{\le 1}(A)\ar@{-}[dr] && \taurigid(A)\ar@{-}[dl]
\\
& \partilt(A) \ar@{-}[d]
\\
& \proj(A)
}
$$
We have
$$
\partilt(A) := \cP_{\le 1}(A) \cap \rigid(A) =
\cP_{\le 1}(A) \cap \taurigid(A).
$$

\subsection{Main results}
For $\cZ \in \irr(A,\bd)$ and $t \ge 1$ let
$$
\cZ^t := \{ M \in \md(A,t \bd) \mid
M \cong M_1 \oplus \cdots \oplus M_t \text{ with }
M_i \in \cZ \text{ for all } i \}.
$$
This is an irreducible constructible subset of $\md(A,t \bd)$.

For $\cZ \in \irr(A)$ let
$$
\pdim(\cZ) := \min\{ \pdim(M) \mid M \in \cZ \}
$$
be the \emph{projective dimension} of $\cZ$.
By upper semicontinuity, the modules
$M \in \cZ$ with $\pdim(M) = \pdim(\cZ)$ form a dense
open subset of $\cZ$.
We also write $\pdim_A(\cZ)$ instead of $\pdim(\cZ)$.

\begin{Qus}
\quad
\begin{itemize}\itemsep2mm

\item[(i)]
Let $M \in \taureg(A)$.
Is $M^t \in \taureg(A)$ for all $t \ge 1$?

\item[(ii)]
Let $\cZ \in \irr^\tau(A)$.
Is $\cZ^t$ generically $\tau$-regular for all $t \ge 1$?

\item[(iii)]
Let $M \in \taureg(A)$.
Is there an ideal $I$ of $A$ with $I M = 0$ such that
$$
\pdim_B(M) \le 1
$$
where $B = A/I$?

\item[(iv)]
Let $\cZ \in \irr^\tau(A)$.
Is there an ideal $I$ of $A$ with $I \cZ = 0$ such that
$$
\pdim_B(\cZ) \le 1
$$
where $B = A/I$?

\end{itemize}
\end{Qus}

All four questions have a negative answer in general.
However, we get positive answers for large classes of algebras.
We will also see that
Questions~(i) and (ii) are strongly related to Questions~(iii) and (iv).

The equivalence (ii) $\iff$ (iii) in
the following theorem
is a direct consequence of Theorem~\ref{thm:BS25a}.

\begin{Thm}\label{thm:intromain1.4}
Let $\calZ \in \irr^\tau(A)$.
Then the following are equivalent:
\begin{itemize}\itemsep2mm

\item[(i)]
$\calZ^t$ is $\tau$-regular for all $t \ge 1$;

\item[(ii)]
For all $M \in \cZ \cap \taureg(A)$,
we have $M^t \in \taureg(A)$ for all $t \ge 1$;

\item[(iii)]
For each $M \in \cZ \cap \taureg(A)$ and a
minimal projective presentation
$$
P_1 \to P_0 \to M \to 0
$$
of $M$, we have
$$
r(P_1^t,P_0^t) = t \cdot r(P_1,P_0)
$$
for all $t \ge 1$.

\end{itemize}
\end{Thm}

The next theorem can be seen as a variation of Theorem~\ref{thm:intromain1.4}.
Note however, that in Theorem~\ref{thm:intromain1.5}(iii) we allow
pairs $(P_1,P_0)$ of projective $A$-modules, which do not appear
in any minimal projective presentation.

\begin{Thm}\label{thm:intromain1.5}
The following are equivalent:
\begin{itemize}\itemsep2mm

\item[(i)]
For all $\cZ \in \irr^\tau(A)$, we have
$\cZ^t$ is $\tau$-regular for all $t \ge 1$;

\item[(ii)]
For all $M \in \taureg(A)$, we have
$M^t \in \taureg(A)$ for all $t \ge 1$;

\item[(iii)]
For all $P_1,P_0 \in \proj(A)$, we have
$$
r(P_1^t,P_0^t) = t \cdot r(P_1,P_0)
$$
for all $t \ge 1$.

\end{itemize}
\end{Thm}

\begin{Thm}\label{thm:intromain1.6}
The following classes of algebras $A$ satisfy the condition
$$
r(P_1^t,P_0^t) = t \cdot r(P_1,P_0)
$$
for all $P_1,P_0 \in \proj(A)$ and $t \ge 1$:
\begin{itemize}\itemsep2mm

\item[(i)]
$\tau$-tilting finite algebras;

\item[(ii)]
Tame algebras;

\item[(iii)]
Algebras $A$ with $\taureg(A) = \cP_{\le 1}(A)$.

\end{itemize}
\end{Thm}

Examples of $\tau$-tilting finite algebras are
representation-finite algebras, dense orbit algebras,
preprojective algebras of Dynkin type, and local algebras.
The algebras mentioned in Theorem~\ref{thm:intromain1.6}(iii) include all hereditary algebras,
and (by a result of Pfeifer \cite[Theorem~1.1]{Pf25}) also the generalized species algebras $H(C,D,\Omega)$ introduced in \cite{GLS17}.

In Section~\ref{subsec:ex4} we present an example of
a wild algebra $A = KQ/I$ with $\gldim(A) = 2$ and some $M \in \md(A)$ such that
$M$ is $\tau$-regular and $M \oplus M$ is neither $\tau$-regular
nor regular.
For a minimal projective presentation
$$
P_1 \to P_0 \to M \to 0
$$
of this module $M$, we have
$$
r(P_1^2,P_0^2) > 2 \cdot r(P_1,P_0)
\text{\quad and \quad}
r(P_1^{2t},P_0^{2t}) = t \cdot r(P_1^2,P_0^2)
$$
for all $t \ge 1$.

We do not know the answer to the following question.

\begin{Qu}
Let $P_1,P_0 \in \proj(A)$.
Is there some $s \ge 1$ such that
$$
r(P_1^{st},P_0^{st}) = t \cdot r(P_1^s,P_0^s)
$$
for all $t \ge 1$?
\end{Qu}

Let $\cZ \in \irr(A)$.
Then $\cZ$ is \emph{sincere}
if it contains a sincere module.

\begin{Thm}\label{thm:intromain1.7}
Let $\calZ \in \irr^\tau(A)$ be sincere.
Then the following are equivalent:
\begin{itemize}\itemsep2mm

\item[(i)]
$\pdim(\calZ) \le 1$;

\item[(ii)]
$\calZ \cap \taureg(A) = \calZ \cap \cP_{\le 1}(A)$.

\end{itemize}
\end{Thm}

For $M \in \md(A)$ let
$$
I_M := \{ a \in A \mid aM = 0 \}.
$$
This is an ideal in $A$.
Then $M$ is \emph{faithful} if $I_M = 0$.

Note that for $M \cong N$ we have $I_M = I_N$.

For each constructible subset $\calU$ of $\md(A,\bd)$ let
$$
I_\calU :=
\{ a \in A \mid aM = 0 \text{ for all }
M \in \calU \} = \bigcap_{M \in \cU} I_M.
$$
This is again an ideal in $A$.
We have
$$
I_{\ov{\cU}} = I_{\cU},
$$
where $\ov{\cU}$ denotes the Zariski closure of $\cU$.

A component $\cZ \in \irr(A)$ is
\emph{faithful} if $I_\cZ = 0$, and
$\cZ$ is
\emph{strongly faithful} if $\cZ$ contains a faithful module.
Strongly faithful components are faithful.
In general, the converse does not hold.
These two types of components where first studied by Mousavand and Paquette \cite{MP23}.

Note that faithful modules and faithful components are always sincere.
The converse does not hold in general.

\begin{Thm}\label{thm:intromain1.8}
Let $\calZ \in \irr^\tau(A)$ be faithful.
Then the following are equivalent:
\begin{itemize}\itemsep2mm

\item[(i)]
$\pdim(\calZ) \le 1$;

\item[(ii)]
$\calZ^t$ is $\tau$-regular for all $t \ge 1$.

\end{itemize}
\end{Thm}

The following proposition is a refinement of
\cite[Proposition~5.2]{MP23}.
It follows essentially from
\cite[Proposition~4.2]{AR77}.

\begin{Prop}\label{prop:introreduction}
Let $\cZ \in \irr^\tau(A)$,
and let $M \in \cZ \cap \taureg(A)$.
Let
$I$ be an ideal of $A$ with $I \cZ = 0$,
and let
$B = A/I$.
Then $M \in \taureg(B)$.
\end{Prop}

Let $\calZ \in \irr(A)$ such that
$\cZ^t$ is $\tau$-regular for all $t \ge 1$.
For each $t \ge 1$ there is a unique $\cZ_t \in \irr^\tau(A)$ such that
$$
\cZ^t \subseteq \cZ_t.
$$
This follows from the fact that all $\tau$-regular modules are regular.

\begin{Thm}\label{thm:intromain1.9}
Let $\calZ \in \irr^\tau(A)$ such that
$\cZ^t$ is $\tau_A$-regular for all $t \ge 1$.
Then the following hold:
\begin{itemize}\itemsep2mm

\item[(i)]
For all $t \ge 1$ we have
$$
I_{\cZ_{t+1}} \subseteq I_{\cZ_t}.
$$

\item[(ii)]
There is some $s(\cZ) \ge 1$ such that
$$
I_{\cZ_\infty} := \bigcap_{t \ge 1} I_{\cZ_t} = I_{\cZ_{s(\cZ)}}.
$$

\item[(iii)]
We have
$I_{\cZ_\infty} \cZ_t = 0$ for all
$t \ge 1$, and $\cZ_t$
is faithful with respect to $B = A/I_{\cZ_\infty}$ for
all $t \ge s(\cZ)$.

\item[(iv)]
$\pdim_B(\cZ_t) \le 1$ for all $t \ge 1$.

\end{itemize}
\end{Thm}

\begin{Cor}\label{cor:intromain1.10}
Let $\cZ \in \irr^\tau(A)$.
Then the following are equivalent:
\begin{itemize}\itemsep2mm

\item[(i)]
$\cZ^t$ is $\tau_A$-regular for all $t \ge 1$;

\item[(ii)]
There exists some ideal $I$ of $A$ with $I \cZ = 0$ such that
$$
\pdim_B(\cZ) \le 1,
$$
where $B = A/I$.

\end{itemize}
In this case,
the ideal $I$ can be chosen such that
$\pdim_B(M) \le 1$ for each $M \in \cZ \cap \taureg(A)$.
\end{Cor}

\subsection{Conventions}
For maps $f\df U \to V$ and $g\df V \to W$ we denote their
composition by $gf\df U \to W$.

The identity map of a set $U$ is denoted by $1_U$.

We assume that the set $\N$ of natural numbers includes $0$.

\subsection{Organization of the article}
In Section~\ref{sec:preliminary} we recall
a few definitions and results from the representation theory
of finite-dimensional algebras.

Section~\ref{sec:directsums}
contains the proofs of
Theorem~\ref{thm:intromain1.4} (= Lemmas~\ref{lem:directsums1}
and \ref{lem:directsums3})
and
Theorem~\ref{thm:intromain1.5} (= Theorem~\ref{thm:directsums4}).

Section~\ref{sec:examples} discusses some classes of
examples.
The content of Sections~\ref{subsec:ex1}, \ref{subsec:ex2},
\ref{subsec:ex3} yields
Theorem~\ref{thm:intromain1.6}.
Section~\ref{subsec:ex4} contains an example of a $\tau$-regular module $M$ such that $M \oplus M$ is not $\tau$-regular.

Section~\ref{sec:faithful} contains the proofs of
Theorem~\ref{thm:intromain1.7} (= Theorem~\ref{thm:sincere}),
Theorem~\ref{thm:intromain1.8} (= Theorem~\ref{thm:faithful4}).

Section~\ref{sec:reduction} deals with
the reduction of $\tau$-rigid and $\tau$-regular modules to
some suitable factor algebras.
It contains the proofs of
Proposition~\ref{prop:introreduction} (= Proposition~\ref{prop:reduction6}),
Theorem~\ref{thm:intromain1.9} (which follows
follows from Lemmas~\ref{lem:limit1} and \ref{lem:limit3} and Theorem~\ref{thm:limit2}) and
Corollary~\ref{cor:intromain1.10} (= Corollary~\ref{cor:limit3}).


\section{Preliminaries}\label{sec:preliminary}


\subsection{Finite-dimensional algebras}
Let $K$ be an algebraically closed field.
By $A$ we always denote a finite-dimensional associative $K$-algebra.
By a \emph{module} we always mean a finite-dimensional left $A$-module.
Let $\md(A)$ be the category of $A$-modules,
and let $\ind(A)$ be the subcategory of indecomposable
$A$-modules.
As usual, $\tau = \tau_A$ denotes the Auslander-Reiten translation
for $\md(A)$.

Let $\proj(A)$ (resp. $\inj(A)$) be the subcategory of projective (resp. injective) $A$-modules.

For $M,N \in \md(A)$ let $\cP(M,N)$ (resp. $\cI(M,N)$)
be the subspace of all $f \in \Hom_A(M,N)$ such that $f$ factors
through some projective (resp. injective) $A$-module.
Define
\begin{align*}
\underline{\Hom}_A(M,N) &:= \Hom_A(M,N)/\cP(M,N),
\\
\ov{\Hom}_A(M,N) &:= \Hom_A(M,N)/\cI(M,N).
\end{align*}
The \emph{stable module category} $\underline{\md}(A)$
(resp. $\ov{\md}(A)$) has the same objects as $\md(A)$
and for $M,N \in \md(A)$ it has the homomorphism
space $\underline{\Hom}_A(M,N)$ (resp. $\ov{\Hom}_A(M,N)$).
The Auslander-Reiten translation $\tau_A$ gives
rise to equivalences
$$
\SelectTips{cm}{}
\xymatrix{
\underline{\md}(A) \ar@/^1ex/[r]^\tau & \ar@/^1ex/[l]^{\tau^-}
\ov{\md}(A)
}
$$
which are mutually quasi-inverses.
Let $D := \Hom_K(-,K)$ be the usual duality.
For $M,N \in \md(A)$, we have the Auslander-Reiten formula
$$
\Ext_A^1(M,N) \cong D\ov{\Hom}_A(N,\tau_A(M))
\cong D\underline{\Hom}_A(\tau_A^{-}(N),M),
$$
see \cite[Section~IV.4]{ARS97} or \cite[Section~III.6]{SY11}
for more details.

For $M,N \in \md(A)$ let
$$
E_A(M,N) := \dim \Hom_A(N,\tau_A(M)),
\text{\quad and \quad}
E_A(M) := E_A(M,M).
$$

For $\cZ,\cZ' \in \irr(A)$ let
\begin{align*}
c_A(\cZ) &:= \min\{ \dim(\cZ) - \dim \cO_M \mid M \in \cZ \},
\\
e_A(\cZ) &:= \min\{ \dim \Ext_A^1(M,M) \mid M \in \cZ \},
\\
e_A(\cZ,\cZ') &:= \min\{ \dim \Ext_A^1(M,M') \mid (M,M') \in \cZ \times \cZ' \},
\\
E_A(\cZ,\cZ') &:= \min\{ E_A(M,M') \mid (M,M') \in \cZ \times \cZ' \}.
\end{align*}

Let $S(1),\ldots,S(n)$ be the simple $A$-modules, up to isomorphism.
Furthermore, let $P(1),\ldots,P(n)$ (resp. $I(1),\ldots,I(n)$)
be the indecomposable projective (resp. indecomposable injective)
$A$-modules, up to isomorphism.
These are numbered such that
$$
\tp(P(i)) \cong S(i) \cong \soc(I(i))
$$
for $1 \le i \le n$.
We have
$$
\dim \Hom_A(P(i),M) = \dim \Hom_A(M,I(i)) = [M:S(i)].
$$

As defined already above,
let
$$
r(M,N) := \max\{ \rk(f) \mid f \in \Hom_A(M,N) \}.
$$

\subsection{Direct sums of irreducible components}
For $1 \le i \le m$ let
$\cZ_i \in \irr(A,\bd_i)$.
Set $\bd := \bd_1 + \cdots + \bd_m$.
We get a morphism
\begin{align*}
\eta\df \GL_\bd \times \cZ_1 \times \cdots \times \cZ_m &\to
\md(A,\bd)
\\
(g,M_1,\ldots,M_m) &\mapsto g.(M_1 \oplus \cdots \oplus M_m).
\end{align*}
The image of $\eta$ is denoted by $\cZ_1 \oplus \cdots \oplus \cZ_m$.
This is an irreducible constructible subset of $\md(A,\bd)$, which is called the \emph{direct sum} of the components
$\cZ_1,\ldots,\cZ_m$.
The direct sum of $t$ copies of $\cZ$ is denoted by $\cZ^t$.
Note that the closure
$$
\ov{\cZ_1 \oplus \cdots \oplus \cZ_m}
$$
is in general not an irreducible component of $\md(A,\bd)$.

\begin{Thm}[{\cite[Theorem~1.2]{CBS02}}]\label{thm:CBS02}
For $1 \le i \le m$ let
$\cZ_i \in \irr(A,\bd_i)$.
Then the following are equivalent
\begin{itemize}\itemsep2mm

\item[(i)]
$\ov{\cZ_1 \oplus \cdots \oplus \cZ_m} \in \irr(A)$;

\item[(ii)]
$e_A(\cZ_i,\cZ_j) = 0$ for all $1 \le i, j \le m$ with $i \not= j$.

\end{itemize}
\end{Thm}

A component $\cZ$ is \emph{indecomposable} if the indecomposable modules
form a dense subset of $\cZ$.
Each component $\cZ$ is (the closure of) a direct sum of indecomposable components in a
unique way, up to permutation of the summands.
This is then called the \emph{canonical decomposition} of $\cZ$.
We refer to \cite{CBS02} for details.

Here is an analogous result for generically $\tau$-regular component.
Its proof is based on Theorem~\ref{thm:CBS02}.

\begin{Thm}[{\cite[Theorem~1.3]{CLFS15}}]\label{thm:CLFS15}
For $1 \le i \le m$ let
$\cZ_i \in \irr^\tau(A,\bd_i)$.
Then the following are equivalent
\begin{itemize}\itemsep2mm

\item[(i)]
$\ov{\cZ_1 \oplus \cdots \oplus \cZ_m} \in \irr^\tau(A)$;

\item[(ii)]
$E_A(\cZ_i,\cZ_j) = 0$ for all $1 \le i, j \le m$ with $i \not= j$.

\end{itemize}
\end{Thm}

The following lemma should be well known.

\begin{Lem}\label{lem:special1}
For $1 \le i \le m$ let
$\cZ_i \in \irr(A,\bd_i)$
such that
$$
\cZ := \ov{\cZ_1 \oplus \cdots \oplus \cZ_m} \in \irr(A).
$$
Then
$$
c_A(\cZ) = \sum_{i=1}^m c_A(\cZ_i)
\text{\quad and \quad}
e_A(\cZ) = \sum_{i=1}^m e_A(\cZ_i).
$$
\end{Lem}

\begin{proof}
Let $\bd := \bd_1 + \cdots + \bd_m$.
We consider the morphism
$$
\eta\df \GL_{\bd} \times \cZ_1 \times \cdots \times \cZ_m \to
\cZ
$$
of affine varieties
which sends $(g,M_1,\ldots,M_m)$ to $g.(M_1 \oplus \cdots \oplus M_m)$.
For $M$ generic in $\cZ$ we can assume that
$M = M_1 \oplus \cdots \oplus M_m$ with $M_i$ generic in $\cZ_i$
for each $1 \le i \le m$.
We get
\begin{align*}
\dim \eta^{-1}(M) &=
\dim \Aut_A(M) + \sum_{i=1}^m \dim \cO_{M_i}.
\end{align*}
It follows that
\begin{align*}
c_A(\cZ) &= \dim(\cZ) - \dim \cO_M
\\
&= \left(\dim \GL_{\bd} + \sum_{i=1}^m \dim(\cZ_i) -
\dim \eta^{-1}(M)\right) -
\dim \cO_M
\\
&=  \left(\dim \GL_{\bd} + \sum_{i=1}^m \dim(\cZ_i)
- \dim \Aut_A(M) -
\sum_{i=1}^m \dim \cO_{M_i}\right) - \dim \cO_M
\\
&= \dim \cO_M + \sum_{i=1}^m c_A(\cZ_i) - \dim \cO_M
\\
&= \sum_{i=1}^m c_A(\cZ_i).
\end{align*}

The equality
$$
e_A(\cZ) = \sum_{i=1}^m e_A(\cZ_i)
$$
follows from Theorem~\ref{thm:CBS02}, which says that
$e_A(\cZ_i,\cZ_j) = 0$ for all $1 \le i, j \le m$ with $i \not= j$.
\end{proof}

\subsection{Derksen and Fei's canonical decomposition of
projective presentations}\label{subsec:DF15}
We recall some results from \cite{DF15}.

For each pair $(P_1,P_0)$ with $P_i \in \proj(A)$ for $i=0,1$
there is a sequence
$$
((P_1^{(1)},P_0^{(1)}), \ldots, (P_1^{(t)},P_0^{(t)}))
$$
of pairs with $P_i^{(j)} \in \proj(A)$ for
$i=0,1$ and $1 \le j \le t$ such that the following hold:
There is a dense open subset $\cU$ of $\Hom_A(P_1,P_0)$ such that
for each $f \in \cU$ there are $f_j \in \Hom_A(P_1^{(j)},P_0^{(j)})$
with $1 \le j \le t$ and an isomorphism
$$
f \cong f_1 \oplus \cdots \oplus f_t
$$
of $2$-complexes, and $f_j$ is an indecomposable $2$-complex for
$1 \le j \le t$.

The above sequence is unique up to permutation of its entries.
We write
$$
(P_1,P_0) = (P_1^{(1)},P_0^{(1)}) \oplus \cdots \oplus
(P_1^{(1)},P_0^{(1)})
$$
and call this the \emph{canonical decomposition} of $(P_1,P_0)$.
This is a slightly modified version of the canonical decomposition introduced and studied in groundbreaking work by
Derksen and Fei \cite{DF15}.

\subsection{Projective presentations of maximal rank and
$\tau$-regular modules}\label{subsec:maximal}
The following result characterizes all $\tau$-regular modules
in terms of their projective presentations.

\begin{Thm}[{\cite[Theorem~1.1]{BS25}}]\label{thm:BS25b}
For $M \in \md(A)$ let
$$
P_1 \xrightarrow{f} P_0 \to M \to 0
$$
be a projective presentation.
Then
the following hold:
\begin{itemize}\itemsep2mm

\item[(i)]
If $\rk(f) = r(P_1,P_0)$, then $M$ is $\tau$-regular;

\item[(ii)]
If $M$ is $\tau$-regular and the presentation above is minimal,
then $\rk(f) = r(P_1,P_0)$.

\end{itemize}
\end{Thm}

Note that Theorem~\ref{thm:BS25b}
is a more detailed version of Theorem~\ref{thm:BS25a}.

\subsection{Direct summands of regular modules and
components}
As a consequence of Theorem~\ref{thm:BS25a}, direct summands of
$\tau$-regular modules and of $\tau$-regular components are again
$\tau$-regular.

Regular modules seem to be much more illusive than $\tau$-regular modules.
We only mention them occasionally.

\begin{Prop}\label{prop:regularsummands1}
For $1 \le i \le m$ let
$\cZ_i \in \irr(A,\bd_i)$
such that
$$
\cZ := \ov{\cZ_1 \oplus \cdots \oplus \cZ_m} \in \irr(A).
$$
Then the following are equivalent:
\begin{itemize}\itemsep2mm

\item[(i)]
$\cZ$ is generically regular;

\item[(ii)]
$\cZ_i$ is generically regular for all $1 \le i \le m$.

\end{itemize}
\end{Prop}

\begin{proof}
By Lemma~\ref{lem:special1}, we have
$$
c_A(\cZ) = \sum_{i=1}^m c_A(\cZ_i)
\text{\quad and \quad}
e_A(\cZ) = \sum_{i=1}^m e_A(\cZ_i).
$$
Since $c_A(\cZ) \le e_A(\cZ)$ and
$c_A(\cZ_i) \le e_A(\cZ_i)$ for all $i$, the claim follows.
\end{proof}

\begin{Qu}\label{qu:regularsummands2}
\quad
\begin{itemize}\itemsep2mm

\item[(i)]
Are direct summands of regular modules again regular?

\item[(ii)]
Let $M \in \md(A)$ such that $M^t$ is regular
for some $t \ge 2$.
Does it follow that $M$ is regular?

\end{itemize}
\end{Qu}


\section{Direct sums of $\tau$-regular modules and components}\label{sec:directsums}


\begin{Lem}\label{lem:regular1}
Let $M \in \taureg(A)$, and let $t \ge 1$.
Then the following are equivalent:
\begin{itemize}\itemsep2mm

\item[(i)]
$M^t \in \taureg(A)$;

\item[(ii)]
$M^t$ is regular.

\end{itemize}
\end{Lem}

\begin{proof}
(i) $\implies$ (ii): Obvious.

(ii) $\implies$ (i):
Assuming that $M^t$ is regular, we get
$$
c_A(M^t) = e_A(M^t) = t^2e_A(M)  = t^2E_A(M) = E_A(M^t).
$$
(The third equality follows from the assumption
that $M$ is $\tau$-regular.
The second and the fourth equality are true for
arbitrary $M \in \md(A)$.)
Thus $M^t \in \taureg(A)$.
\end{proof}

For all $P_1,P_0 \in \proj(A)$ and $t \ge 1$ we obviously have
$$
r(P_1^t,P_0^t) \ge t \cdot r(P_1,P_0).
$$

\begin{Lem}\label{lem:directsums3}
Let $M \in \taureg(A)$, and
let
$$
P_1 \to P_0 \to M \to 0
$$
be a minimal projective presentation
of $M$.
Then the following are equivalent:
\begin{itemize}\itemsep2mm

\item[(i)]
$M^t \in \taureg(A)$ for all $t \ge 1$;

\item[(ii)]
$r(P_1^t,P_0^t) = t \cdot r(P_1,P_0)$
for all $t \ge 1$.

\end{itemize}
\end{Lem}

\begin{proof}
This follows directly from Theorem~\ref{thm:BS25a}.
\end{proof}

For $f \in \Hom_A(M,N)$ and $t \ge 1$, let
$$
f^{\oplus t} := \left(\bbm f &0&0\\0&\ddots&0\\0&0&f
\ebm\right)\df M^t \to N^t
$$
be the $t$-fold direct sum of $f$.

\begin{Lem}\label{lem:directsums1}
Let $\cZ \in \irr(A)$.
Then the following are equivalent:
\begin{itemize}\itemsep2mm

\item[(i)]
$\cZ^t$ is $\tau$-regular for all $t \ge 1$;

\item[(ii)]
For all $M \in \cZ \cap \taureg(A)$, we have
$M^t \in \taureg(A)$ for all $t \ge 1$.

\end{itemize}
\end{Lem}

\begin{proof}
(ii) $\implies$ (i):
This follows directly from the definitions.

(i) $\implies$ (ii):
Let $M \in \cZ \cap \taureg(A)$, and let
$$
P_1 \xrightarrow{f_M} P_0 \to M \to 0
$$
be a minimal projective presentation of $M$.
By Theorem~\ref{thm:BS25a} we have
$\rk(f_M) = r(P_1,P_0)$.
We can choose some generic $N$ in $\cZ^t$ such
that $N = N_1 \oplus \cdots \oplus N_t$ with $N_i$
generic in $\cZ$ for each $1 \le i \le t$.
In particular, $N$ and all $N_i$ are $\tau$-regular.
By Fei's Lemma \cite[Lemma~3.2]{BS25},
we can assume that there is a projective presentation
$$
P_1 \xrightarrow{f_i} P_0 \to N_i \to 0
$$
for each $1 \le i \le t$.
The $2$-complex
$$
P_1 \xrightarrow{f_i} P_0
$$
is isomorphic to a direct sum
$$
(P_1' \xrightarrow{f_i'} P_0')
\oplus (P \xrightarrow{1_P} P)
\oplus (P' \to 0)
$$
of $2$-complexes of project $A$-modules
where
$$
P_1' \xrightarrow{f_i'} P_0' \to N_i \to 0
$$
is a minimal projective presentation, see \cite[Lemma~2.5]{BS25}.
Note that $\rk(f_i) = \rk(f_M) = r(P_1,P_0)$.
Then
\cite[Lemma~3.6]{BS25} says that
$\rk(f_i') = r(P_1',P_0')$ and
$\Hom_A(P',N_i) = 0$.
Thus we also have $\Hom_A(P',N) = 0$.

Since $N_i$ is generic in $\cZ$ for $1 \le i \le t$,
the pairs $(P_1',P_0')$, $(P,P)$ and $(P',0)$ do not depend on $i$.
(This follows from the existence of the canonical decomposition
of $(P_1,P_0)$ in the sense of Derksen and Fei \cite{DF15}
(see Section~\ref{subsec:DF15}) and the fact that
generic morphisms in $\Hom_A(P_1,P_0)$ correspond to generic
modules in $\cZ$.
For more details on this correspondence, we refer to
\cite[Section~5]{BS25}, which contains an account of Plamondon's
\cite{P13} construction of generically $\tau$-regular irreducible components.)

We get
a projective presentation
$$
P_1^t \xrightarrow{f} P_0^t \to N \to 0
$$
where
$$
f = \left(\bbm f_1 & 0&0\\
0 & \ddots & 0\\
0&0 &f_t \ebm\right),
$$
such that the $2$-complex
$$
P_1^t \xrightarrow{f} P_0^t
$$
is isomorphic to
$$
(P_1'^t \xrightarrow{f'} P_0'^t)
\oplus (P^t \xrightarrow{1_{P^t}} P^t)
\oplus (P'^t \to 0)
$$
where
$$
f' = \left(\bbm f_1' & 0&0\\
0 & \ddots & 0\\
0&0 &f_t' \ebm\right).
$$
Note that
$$
P_1'^t \xrightarrow{f'} P_0'^t \to N \to 0
$$ is a minimal projective presentation.
Since $N$ is $\tau$-regular, Theorem~\ref{thm:BS25a} implies
$\rk(f') = r(P_1'^t,P_0'^t)$.
Now \cite[Lemma~3.6]{BS25} yields
$$
r(P_1^t,P_0^t) = \rk(f) = \sum_{i=1}^t \rk(f_i) = t \cdot r(P_1,P_0) = t \cdot \rk(f_M) = \rk(f_M^{\oplus t}).
$$
Thus by Theorem~\ref{thm:BS25a}, $M^t$ is $\tau$-regular.
\end{proof}

\begin{Lem}\label{lem:directsums2}
Let $\cZ \in \irr^\tau(A)$ with $\pdim(\cZ) \le 1$.
Then $\cZ^t$ is $\tau$-regular for all $t \ge 1$.
\end{Lem}

\begin{proof}
Let $M \in \cZ$ with $\pdim(M) \le 1$.
Then $\pdim(M^t) \le 1$ for all $t \ge 1$.
Thus $M^t$ and therefore also $\cZ^t$
is $\tau$-regular for all $t \ge 1$.
\end{proof}

\begin{Thm}\label{thm:directsums4}
The following are equivalent:
\begin{itemize}\itemsep2mm

\item[(i)]
For all $\cZ \in \irr^\tau(A)$, we have
$\cZ^t$ is $\tau$-regular for all $t \ge 1$;

\item[(ii)]
For all $M \in \taureg(A)$, we have
$M^t \in \taureg(A)$ for all $t \ge 1$;

\item[(iii)]
For all $P_1,P_0 \in \proj(A)$, we have
$$
r(P_1^t,P_0^t) = t \cdot r(P_1,P_0)
$$
for all $t \ge 1$.

\end{itemize}
\end{Thm}

\begin{proof}
(i) $\iff$ (ii):
Each $M \in \taureg(A)$ is contained in a (unique)
$\cZ \in \irr^\tau(A)$.
Thus the claim
follows directly from Lemma~\ref{lem:directsums1}.

(iii) $\implies$ (ii):
This follows from Lemma~\ref{lem:directsums3}.

(ii) $\implies$ (iii):
Let $P_1,P_0 \in \proj(A)$, and let
$$
P_1 \xrightarrow{f} P_0 \to M \to 0
$$
be a projective presentation of some $M$ with $f$ generic in
$\Hom_A(P_1,P_0)$.
In particular, we have $\rk(f) = r(P_1,P_0)$.
By Theorem~\ref{thm:BS25b}(i), we get that $M$ is $\tau$-regular.

We get an isomorphism
$$
(P_1 \xrightarrow{f} P_0) \cong
(P_1^M \xrightarrow{f_M} P_0^M) \oplus
(P \xrightarrow{1_P} P) \oplus
(P_f \to 0)
$$
of $2$-complexes of projective $A$-modules, where
$$
P_1^M \xrightarrow{f_M} P_0^M \to M \to 0
$$
is a minimal projective presentation of $M$ and
$\Hom_A(P_f,M) = 0$, compare \cite[Lemmas~2.5 and 3.6]{BS25}.

Since $\rk(f)$ is maximal, we have
$r(P_1,P_0) = \rk(f) = r(P_1^M,P_0^M) + r(P,P)$.

Then
$$
(P_1^M)^t \xrightarrow{f_M^{\oplus t}} (P_0^M)^t \to M^t \to 0
$$
is a minimal projective presentation of $M^t$.

We get an isomorphism
$$
(P_1^t \xrightarrow{f^{\oplus t}} P_0^t) \cong
((P_1^M)^t \xrightarrow{f_M^{\oplus t}} (P_0^M)^t) \oplus
(P^t \xrightarrow{1_{P^t}} P^t) \oplus
(P_f^t \to 0)
$$
of $2$-complexes of projective $A$-modules,

By (ii), the module $M^t$ is $\tau$-regular.
Thus by Theorem~\ref{thm:BS25a}, we have
$$
\rk(f_M^{\oplus t}) = r((P_1^M)^t,(P_0^M)^t).
$$
Moreover, $\Hom (P_f^t, M^t) = 0$.

Now \cite[Lemma~3.6]{BS25} (applied to $f^{\oplus t}$) says that
\[
t \cdot r(P_1,P_0) = \rk(f^{\oplus t}) = r(P_1^t,P_0^t).
\qedhere
\]
\end{proof}


\section{Examples}\label{sec:examples}


The aim of this section is to
prove Theorem~\ref{thm:intromain1.6}, which names
some classes of algebras $A$ such that
\begin{equation}\label{eq:maxrank}
r(P^t,P_0^t) = t \cdot r(P_1,P_0)
\text{ for all }
P_1,P_0 \in \proj(A) \text{ and } t \ge 1. \tag{$\star$}
\end{equation}
Recall that this
condition appears also in Theorem~\ref{thm:intromain1.5}.

We also construct one example, where Condition~(\ref{eq:maxrank}) does not hold.

\subsection{$\tau$-tilting finite algebras}\label{subsec:ex1}
Assume that
$A$ is \emph{$\tau$-tilting finite}, i.e.\ there are only
finitely many indecomposable $\tau$-rigid $A$-modules, up to isomorphism.
These algebras were defined and studied by
Demonet, Iyama and Jasso \cite{DIJ19}.

Then each $\cZ \in \irr^\tau(A)$ is of the form
$$
\cZ = \ov{\cO_M}
$$
for some $M \in \taurigid(A)$.
(This follows by combining results in \cite{DIJ19} and
\cite{A21}.)

We obviously have $M^t \in \taurigid(A)$ for all such $M$ and
$t \ge 1$.
Thus Theorem~\ref{thm:intromain1.5} says that Condition~(\ref{eq:maxrank}) holds.

The following are examples of $\tau$-tilting finite algebras.
\begin{itemize}\itemsep2mm

\item[(i)]
Representation-finite algebras;

\item[(ii)]
Dense orbit algebras (as defined in \cite{CKW15});

\item[(iii)]
The generalized species algebras $H(C,D,\Omega)$
of Dynkin type (see \cite{GLS17});

\item[(iv)]
Preprojective algebra of Dynkin type;

\item[(v)]
Local algebras.

\end{itemize}

\subsection{Tame algebras}\label{subsec:ex2}
\quad
\begin{Thm}\label{thm:tame2}
Assume that $A$ is tame, and let
$\calZ \in \irr(A)$ be generically regular.
Then the following hold:
\begin{itemize}\itemsep2mm

\item[(i)]
$e_A(\cZ,\cZ) = 0$;

\item[(ii)]
$\ov{\calZ^t} \in \irr(A)$ is generically regular for all $t \ge 1$.

\end{itemize}
\end{Thm}

\begin{proof}
Let
$$
\cZ = \ov{\cZ_1 \oplus \cdots \cZ_m}
$$
be the canonical decomposition of $\cZ$.
Since $A$ is tame, we have $c_A(\cZ_i) \le 1$ for all $1 \le i \le m$.
Recall that $e_A(\cZ_i,\cZ_j)=0$ for all $i \neq j$.

Assume that $\cZ$ is generically regular.
Then
$$
\sum_{i=1}^m e_A(\cZ_i) = e_A(\cZ) =
c_A(\cZ) = \sum_{i=1}^m c_A(\cZ_i).
$$
(The first and third equality follows from Lemma~\ref{lem:special1}.)
In particular, all $\cZ_i$ are generically regular, and we have
$e_A(\cZ_i) \le 1$.
By \cite[Theorem~1.5]{GLFS24} we get $e_A(\cZ_i,\cZ_i) < e_A(\cZ_i)$
if $e_A(\cZ_i) = 1$, or $e_A(\cZ_i,\cZ_i) = e_A(\cZ_i) = 0$, otherwise.
Thus in both cases, we have $e_A(\cZ_i,\cZ_i) = 0$ and therefore
$$
\cZ' := \ov{\cZ^t} = \ov{\cZ_1^t \oplus \cdots \oplus \cZ_m^t}
\in \irr(A).
$$
We get
$$
c_A(\cZ') = tc_A(\cZ) = te_A(\cZ) = e_A(\cZ').
$$
Thus $\cZ'$ is generically regular.
\end{proof}

What follows is a corresponding statement for generically $\tau$-regular components.

\begin{Thm}\label{thm:tame1}
Assume that $A$ is tame, and let
$\calZ \in \irr^\tau(A)$.
Then the following hold:
\begin{itemize}\itemsep2mm

\item[(i)]
$E_A(\cZ,\cZ) = 0$;

\item[(ii)]
$\ov{\calZ^t} \in \irr^\tau(A)$ for all $t \ge 1$.

\end{itemize}
In particular, $\cZ^t$ is $\tau$-regular for all $t \ge 1$.
\end{Thm}

\begin{proof}
(i):
This is part of \cite[Theorem~3.8]{PYK23}, compare also
\cite[Corollary~1.7]{GLFS24}.
The proof uses a classical deep result by Crawley-Boevey \cite[Theorem~D]{CB88} on tame algebras.

(ii)
This follows from (i) combined with
\cite[Theorem~1.3]{CLFS15}.
\end{proof}

Note that Theorem~\ref{thm:tame1} implies that all tame algebra satisfy Condition~(\ref{eq:maxrank}).

\subsection{Algebras with $\taureg(A) = \cP_{\le 1}(A)$}\label{subsec:ex3}
Let $A$ be an algebra such that
$$
\taureg(A) = \cP_{\le 1}(A).
$$
Then for each $M \in \taureg(A)$, we obviously have
$M^t \in \taureg(A)$.
Thus by Theorem~\ref{thm:intromain1.5}, $A$ satisfies Condition~(\ref{eq:maxrank}).

The following are examples of algebras $A$ satisfying
$$
\taureg(A) = \cP_{\le 1}(A).
$$
\begin{itemize}\itemsep2mm

\item[(i)]
Hereditary  algebras, i.e.\ algebras with
$\gldim(A) \le 1$.
(This is equivalent to
$\md(A) = \cP_{\le 1}(A)$.)

\item[(ii)]
The generalized species algebras
$H(C,D,\Omega)$ introduced in \cite{GLS17}.
(For these the condition $\taureg(A) = \cP_{\le 1}(A)$ follows from
\cite[Theorem~1.1]{Pf25}.)

\item[(iii)]
Local algebras.
(For these one has
$\taureg(A) = \proj(A) = \cP_{\le 1}(A)$.)

\end{itemize}

\subsection{A counterexample}\label{subsec:ex4}
Let $A = KQ / I$, where $Q$ is the quiver
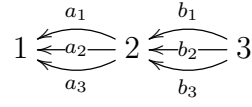
\[
\xymatrix@+1ex{1 & 2 \ar@/_1.5ex/[l]_{a_1} \ar[l]|{a_2} \ar@/^1.5ex/[l]^{a_3} & 3 \ar@/_1.5ex/[l]_{b_1} \ar[l]|{b_2} \ar@/^1.5ex/[l]^{b_3}}
\]
and $I$ is generated by the elements
$a_i b_i$ with $i = 1, 2, 3$ and
\[
a_1b_2 - a_2b_1,\quad
a_1b_3 + a_3b_1,\quad
a_2b_3 - a_3b_2.
\]

The sign convention in the second set of relations is
relevant for our calculations below.
We obviously have $P(1) = S(1)$.
We have
$$
\dimv(P(1)) = (1,0,0),\quad
\dimv(P(2)) = (3,1,0),\quad
\dimv(P(3)) = (3,3,1).
$$
The modules $P(2)$ and $P(3)$ can be visualized as follows:
$$
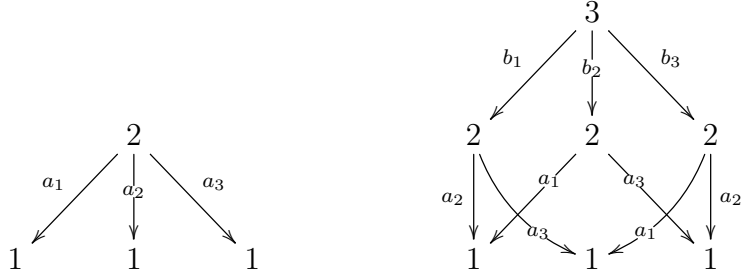
\xymatrix@+1.5ex{
&&&&& 3 \ar[dl]_{b_1}\ar[d]|<<<<<{b_2}\ar[dr]^{b_3}
\\
&2 \ar[dl]_{a_1}\ar[d]|<<<<<{a_2}\ar[dr]^{a_3}
&&&
2\ar[d]_{a_2}\ar@/_1.6ex/[dr]|>>>>{a_3} &
2\ar[dl]|<<<<<{a_1}
\ar[dr]|<<<<<{a_3} & 2 \ar[d]^{a_2}\ar@/^1.6ex/[dl]|>>>>{a_1}
\\
1 & 1& 1
&&
1 & 1 & 1
}
$$
(Each number $i$ stands both for a composition factor $S(i)$
and for a basis vector.
The arrows indicate how the arrows
of the quiver $Q$ act on these basis vectors.
They send a basis
vector $i$ to $0$ or to $\pm(i-1)$
for some other basis vector $i-1$.
The signs are chosen such that the defining relations for $I$ are
respected.)

Note that $\gldim(A) = 2$.

\noindent
{\bf Claim 1}: $r(P(2),P(3)) = 3$.

\begin{proof}
Let $f \in \Hom_A (P(2), P(3))$.
Then
$$
f (e_2) = \lambda_1 b_1 + \lambda_2 b_2 + \lambda_3 b_3
$$
for some $(\lambda_1,\lambda_2,\lambda_3) \in K^3$,
and these scalars are uniquely determined by $f$.
Set $f^{(\lambda_1, \lambda_2, \lambda_3)} := f$.
Then
\[
f (a_1) = \lambda_2 a_1 b_2 + \lambda_3 a_1 b_3, \quad f (a_2) = \lambda_1 a_1 b_2 + \lambda_3 a_2 b_3, \qquad f (a_3) = - \lambda_1 a_1 b_3 + \lambda_2 a_2 b_3.
\]
Consequently, we obtain
$\lambda_1 f (a_1) - \lambda_2 f (a_2) + \lambda_3 f (a_3) = 0$.
\end{proof}

\noindent
{\bf Claim 2}: $r(P(2)^2,P(3)^2) = 8 > 2 \cdot r(P(2),P(3))$.

\begin{proof}
Let $g \in \Hom_A (P(2)^2, P(3)^2)$ be given by the matrix
\[
\begin{pmatrix}
f^{(1, 0, 0)} & f^{(0, 1, 0)} \\ f^{(0, 1, 0)} & f^{(0, 0, 1)}
\end{pmatrix}.
\]
An easy calculation yields that
$\rk(g) = 8 = r(P(2)^2,P(3)^2)$.
\end{proof}

Fix $(\lambda_1, \lambda_2, \lambda_3) \neq (0, 0, 0)$ and put
$M := \Coker f^{(\lambda_1, \lambda_2, \lambda_3)}$.

Combining Claims 1 and 2, we obtain the following:
\begin{itemize}\itemsep2mm

\item[(i)]
$M$ is $\tau$-regular;

\item[(ii)]
$M \oplus M$ is not $\tau$-regular (and therefore by Lemma~\ref{lem:regular1} also not regular).

\end{itemize}

It would be interesting to find larger classes of algebras and modules
with this behaviour.

Note that the homomorphism
$g\df P(2)^2 \to P(3)^2$ (as constructed in the proof of Claim 2)
is a monomorphism.
Therefore we have
$$
\rk(g) = r(P(2)^2,P(3)^2)
\text{\quad and \quad}
r(P(2)^{2t},P(3)^{2t}) = t \cdot r(P(2)^2,P(3)^2)
$$
for all $t \ge 1$.


\section{Sincere and faithful $\tau$-regular components}
\label{sec:faithful}


\subsection{Sincere components of projective dimension at most one}
\quad
\begin{Thm}\label{thm:sincere}
Let $\calZ \in \irr^\tau(A)$ be sincere.
Then the following are equivalent:
\begin{itemize}\itemsep2mm

\item[(i)]
$\pdim(\calZ) \le 1$;

\item[(ii)]
$\calZ \cap \taureg(A) = \calZ \cap \cP_{\le1}(A)$.

\end{itemize}
\end{Thm}

\begin{proof}
(i) $\implies$ (ii):
Assume that $\pdim(\cZ) \le 1$.
We always have $\cP_{\le 1}(A) \subseteq \taureg(A)$.
Assume now $M \in \cZ \cap \taureg(A)$.
Let
$$
P_1 \xrightarrow{f}  P_0 \to M \to 0
$$
be a minimal projective presentation of $M$.
We know from Fei's Lemma \cite[Lemma~3.2]{BS25} that there is a generic $N \in \cZ$, which has
a projective presentation of the form
$$
P_1 \xrightarrow{g}  P_0 \to N \to 0.
$$
Since $\pdim(N) \le 1$, we get an isomorphism
$$
(P_1 \xrightarrow{g} P_0) \cong
(P_1' \xrightarrow{g'} P_0') \oplus
(P \xrightarrow{\cong} P) \oplus
(P' \to 0)
$$
of $2$-complexes of projective $A$-modules, where $g'$ is a minimal projective presentation of
$N$.
In particular, $g'$ is a monomorphism.
We also know that $\rk(g) = \rk(f) = r(P_1,P_0)$, since $M$ is $\tau$-regular.
This implies $\Hom_A(P',N) = 0$.
Since $\cZ$ and therefore also $N$ is sincere, this implies
$P' = 0$.
It follows that $g$ and therefore also $f$ is a monomorphism.
In particular, $\pdim(M) \le 1$.

(ii) $\implies$ (i):
Obvious.
\end{proof}

\Example:
Let $A = KQ/I$ where
$Q$ is the quiver
$$
\xymatrix{
1 & 2\ar[l]_a & 3\ar[l]_b
}
$$
and $I$ is generated by $ab$, and let
$\bd = (0,1,1)$.
Then $\calZ := \md(A,\bd)$ is a non-sincere irreducible component.
We have
$$
\calZ = \ov{\cO_{P(3)}}.
$$
Thus $\calZ$ is $\tau$-regular with $\pdim(\calZ) = 0 \le 1$.
Let $M := S(2) \oplus S(3)$.
One easily checks that $M$ is $\tau$-regular and that
$\pdim(M) = 2$.
Thus Theorem~\ref{thm:sincere} does not generalize to non-sincere
components.

\subsection{Faithful components}
Mousavand and Paquette \cite{MP23} initiated the study of faithful
irreducible components.
We recall and generalize some of their results.

\begin{Prop}[{\cite[Proposition~5.3]{MP23}}]\label{prop:faithful1}
Let $\calZ \in \irr(A)$.
Then
$$
\{ M \in \calZ \mid M \text{ is faithful} \}
$$
is an open subset of $\calZ$.
\end{Prop}

\begin{Prop}\label{prop:finiteintersection}
Let $\calZ \in \irr(A)$, and let $\cU$ be a dense open subset
of $\calZ$.
Then there exist $M_1,\ldots,M_t \in \cU$ such that
$$
I_\calZ = \bigcap_{1 \le i \le t} I_{M_i} = I_{M_1 \oplus \cdots \oplus M_t}.
$$
\end{Prop}

\begin{proof}
We have
$$
I_\calZ = I_\cU = \bigcap_{M \in \cU} I_M.
$$
Since the ideals $I_M$ are subspaces of the finite-dimensional
vector space $A$, there exist finitely many modules $M_1,\ldots,M_t \in \cU$
with 
$$
I_\cU = \bigcap_{1 \le i \le t} I_{M_i}.
$$
Finally, it follows from the definitions that
\[
\bigcap_{1 \le i \le t} I_{M_i} = I_{M_1 \oplus \cdots \oplus M_t}.
\qedhere
\]
\end{proof}

\begin{Cor}[{\cite[Proposition~5.4]{MP23}}]\label{cor:faithful2}
Let $\calZ \in \irr(A)$ be faithful. Then there exist generic modules
$M_1,\ldots,M_r$ in $\calZ$ such that
$M_1 \oplus \cdots \oplus M_t$ is faithful.
\end{Cor}

\begin{proof}
Combine Proposition~\ref{prop:finiteintersection} with Proposition~\ref{prop:faithful1}.
\end{proof}

\begin{Lem}
For $\cZ,\cZ_1,\ldots,\cZ_m \in \irr(A)$ with
$$
\cZ = \ov{\cZ_1 \oplus \cdots \oplus \cZ_m}
$$
we have
$$
I_\cZ = \bigcap_{1 \le i \le m} I_{\cZ_i}.
$$
\end{Lem}

\begin{proof}
This is obvious, since
\[
I_\cZ = I_{\cZ_1 \oplus \cdots \oplus \cZ_m}.
\qedhere
\]
\end{proof}

\Remark:
It should be interesting to study the factor algebras
$A/I_\cZ$ where $\cZ \in \irr^\tau(A)$. 
For $A$ a preprojective algebra, some results in this direction
can be found in \cite{AIRT12}.

\begin{Prop}[{\cite[Proposition~5.5]{MP23}}]\label{prop:faithful3}
Let $\calZ \in \irr^\tau(A)$ be faithful.
If $\calZ$ is strongly faithful, or if $E_A(\calZ,\calZ) = 0$,
then $\pdim(\cZ) \le 1$.
\end{Prop}

Proposition~\ref{prop:faithful3} gives a criterium when a faithful
$\tau$-regular component is of projective dimension at most one.
The following theorem gives a necessary and sufficient condition for this.

\begin{Thm}\label{thm:faithful4}
Let $\calZ \in \irr^\tau(A)$ be faithful.
Then the following are equivalent:
\begin{itemize}\itemsep2mm

\item[(i)]
$\pdim(\calZ) \le 1$;

\item[(ii)]
$\calZ^t$ is $\tau_A$-regular for all $t \ge 1$.

\end{itemize}
\end{Thm}

\begin{proof}
(i) $\implies$ (ii):
This is part of Lemma~\ref{lem:directsums2}.

(ii) $\implies$ (i):
Since $\cZ$ is faithful, we can apply Corollary~\ref{cor:faithful2}
and get generic modules $M_1,\ldots,M_t$ in $\cZ$ such that
$M := M_1 \oplus \cdots \oplus M_t$ is faithful.
We have $M \in \cZ^t$.
The faithful modules form a non-empty open subset of
$\cZ^t$, and the same holds for the $\tau$-regular modules.
Thus there exists a faithful $\tau$-regular module $N \in \cZ^t$.
Therefore there is some $r \ge 1$ and an epimorphism
$f\df N^r \to D(A_A)$.
Applying $\Hom_A(-,\tau_A(N))$ yields an exact sequence
$$
0 \to \Hom_A(D(A_A),\tau_A(N)) \xrightarrow{h}
\Hom_A(N^r,\tau_A(N)).
$$
where $h := \Hom_A(f,\tau_A(N))$.
The image of $h$ consists of morphisms factoring through the
injective module $D(A_A)$.

The $\tau$-regularity of $N$ implies that $\Ext_A^1(N,N) \cong
\Hom_A(N,\tau_A(N))$, and the Auslander-Reiten formulas say
that $\Ext_A^1(N,N) \cong \ov{\Hom}_A(N,\tau_A(N))$.
Thus we have
$$
\ov{\Hom}_A(N^r,\tau_A(N)) = \Hom_A(N^r,\tau_A(N)).
$$
This implies that the image of $h$ is $0$.
Thus we obtain
$\Hom_A(D(A_A),\tau_A(N)) = 0$, and therefore $\pdim(N) \le 1$.
(Note that we use here the same argument as in Mousavand and Paquette's proof of
Proposition~\ref{prop:faithful3}.)
We have $N = N_1 \oplus \cdots \oplus N_t$ with $N_i \in \cZ$
for $1 \le i \le t$.
We get $\pdim(N_i) \le 1$  for all $i$ and therefore
$\pdim(\cZ) \le 1$.
\end{proof}


\section{Reduction of $\tau$-regular modules and components}
\label{sec:reduction}


\subsection{Reduction of modules}
Let $I$ be an ideal of $A$, and let $B = A/I$.
Then we can identify $\md(B)$ with the full subcategory of all
$M \in \md(A)$ with $IM = 0$.
When we consider such an $A$-module $M$ as a $B$-module,
we speak of the \emph{reduction} of $M$ to $B$ (or to a $B$-module).
Similarly, if $\cZ \in \irr(A)$ with $I \cZ = 0$, then we can consider
$\cZ$ as a component in $\irr(B)$, and we speak
again of a \emph{reduction} of $\cZ$ to $B$ (or to a component
in $\irr(B)$).

\begin{Lem}\label{lem:reduction5}
As above, let $B = A/I$, and let $M \in \md(A)$ with $IM = 0$.
Then
$$
c_B(M) \le c_A(M), \quad
e_B(M) \le e_A(M), \quad
E_B(M) \le E_A(M).
$$
\end{Lem}

\begin{proof}
We have $M \in \md(A,\bd)$ with $\bd = \dimv(M)$.
We can interpret $\md(B,\bd)$ as a closed subset of
$\md(A,\bd)$.
This implies $c_B(M) \le c_A(M)$.

Note that $\md(B)$ is an exact subcategory of $\md(A)$.
It follows that $e_B(M) \le e_A(M)$.

By \cite[Proposition~4.2]{AR77} we know that $\tau_B(M)$ can be seen
as a submodule of $\tau_A(M)$.
This obviously implies
$$
\Hom_B(M,\tau_B(M)) = \Hom_A(M,\tau_B(M)) \subseteq
\Hom_A(M,\tau_A(M))
$$
and therefore $E_B(M) \le E_A(M)$.
\end{proof}

\subsection{Reduction of rigid and $\tau$-rigid modules}

\begin{Lem}\label{lem:reduction1}
As above, let $B = A/I$, and let $M \in \md(A)$ with $IM = 0$.
If
$M$ is rigid (resp. $\tau_A$-rigid) as an $A$-module,
then  $M$ is rigid (resp. $\tau_B$-rigid) as a $B$-module.
\end{Lem}

\begin{proof}
For rigid modules the claim is straightforward, since
$\md(B)$ is an exact subcategory of $\md(A)$.
For $\tau$-rigid modules, the claim
follows directly from
Proposition~\ref{lem:reduction5}.
\end{proof}

The previous lemma does not generalize to $\tau$-regular modules,
as Example~(i) in Section~\ref{subsec:reduction1} shows.

The following lemma seems to be folklore.

\begin{Lem}[{\cite[VIII, Lemma~5.1]{ASS06}}]\label{lem:reduction2}
Let $M \in \md(A)$ be faithful and $\tau$-rigid.
Then
$$
\pdim_A(M) \le 1.
$$
\end{Lem}

\begin{Cor}\label{cor:reduction3}
Let $M \in \md(A)$ be $\tau$-rigid, and let $B = A/I_M$.
Then
$$
\pdim_B(M) \le 1.
$$
\end{Cor}

In other words, $\tau$-rigid $A$-modules become modules of projective dimension at most one over some factor algebras of $A$.
We want to explore if there is at least a partial analogous result for
$\tau$-regular modules.

\subsection{Reduction of $\tau$-regular modules}
\label{subsec:reduction1}

The following proposition due to Mousavand and Paquette says that the reduction of a generically $\tau$-regular
component is always a generically $\tau$-regular component.

\begin{Prop}[{\cite[Proposition~5.2]{MP23}}]\label{prop:reduction4}
Let $\calZ \in \irr^\tau(A)$, and let $I$ be an ideal of $A$ with
$I \calZ = 0$, and let $B = A/I$.
Then $\calZ \in \irr^\tau(B)$.
\end{Prop}

The following is a refinement of Proposition~\ref{prop:reduction4}.

\begin{Prop}\label{prop:reduction6}
Let $\cZ \in \irr^\tau(A)$,
and let $M \in \cZ \cap \taureg(A)$.
Let
$I$ be an ideal of $A$ with $I \cZ = 0$,
and let
$B = A/I$.
Then $M \in \taureg(B)$.
\end{Prop}

\begin{proof}
We have $\cZ \in \irr(B)$.
This implies $c_B(M) = c_A(M)$.
We get
$$
E_B(M) \le E_A(M) = c_A(M) = c_B(M).
$$
(The first inequality follows from Lemma~\ref{lem:reduction5}.)
Using Proposition~\ref{prop:voigt},
this implies now $c_B(M) = E_B(M)$, i.e.\ $M \in \taureg(B)$.
\end{proof}

\Examples:
\begin{itemize}\itemsep2mm

\item[(i)]
In Proposition~\ref{prop:reduction6}, it is not enough to take an ideal $I$
of $A$ with $I M = 0$, as the following example shows.
Let $A = KQ$ and $B = KQ/I$ where $Q$ is the quiver
$$
\xymatrix{
1 & 2 \ar[l]_a & 3 \ar[l]_b
}
$$
and $I$ is generated by $ab$.
Let $M = P(2) \oplus I(2) \oplus S(3)$.
Then we have $I_M = I$.
In particular, $I M = 0$.
Observe also that $I$ is the only non-zero ideal of $A$
with $IM = 0$.
Then $M$ has the following properties:
\begin{itemize}\itemsep2mm

\item[$\bullet$]
$M$ is not $\tau_A$-rigid;

\item[$\bullet$]
$\pdim_A(M) = 1$ and $\pdim_B(M) = 2$;

\item[$\bullet$]
$M$ is $\tau_A$-regular but not $\tau_B$-regular.

\end{itemize}

\item[(ii)]
As the following example shows, the reduction of a generically
$\tau_A$-regular component $\cZ$
to a generically $\tau_B$-regular component does in general not lead to a bijection
between the $\tau_A$-regular modules and
the $\tau_B$-regular modules in $\cZ$.
In other words, the converse of Proposition~\ref{prop:reduction6} does not hold.
Let $A = KQ/I$ where $Q$ is the quiver
$$
\xymatrix@+1ex{
1 \ar@/_1.5ex/[r]_a & 2 \ar@/_1.5ex/[l]\ar[l]
}
$$
and $I$ is generated by all paths of length $2$.
Let $\bd = (1,1)$, and let
$\cZ = \{ M \in \md(A,\bd) \mid aM = 0 \}$.
Then $\cZ \in \irr^\tau(A)$ with $\pdim_A(\cZ) = \infty$.
The only non-zero ideal $J$ in $A$ with $J \cZ = 0$ is
the $1$-dimensional ideal $J = (a)$ generated by the arrow $a$.
Using the notation from Theorem~\ref{thm:intromain1.9} we have
$$
\ov{\cZ^t} = \cZ_t
\text{\quad and \quad}
(a) = I_{\cZ_\infty} = I_{\cZ_t}
$$
for all $t \ge 1$.
Let $B = A/I_{\cZ_\infty}$.
Then $B$ is isomorphic to the path algebra
of the Kronecker quiver
$$
\xymatrix{
1 & 2 \ar@/^1ex/[l]\ar@/_1ex/[l]
}.
$$
In particular, $B$ is hereditary, and therefore all $B$-modules
are $\tau_B$-regular.
The $A$-module $N = S(1) \oplus S(2)$ is contained in $\cZ$.
It is $\tau_B$-regular but not $\tau_A$-regular.
In fact, $N$ is the only module in $\cZ$, which is not $\tau_A$-regular.

\end{itemize}

\subsection{Reduction to components of projective dimension
at most one}
\label{subpsec:pdim1}

Let $\cZ \in \irr^\tau(A)$.
Assume that there is some $M \in \cZ$ such that $M^t$ is $\tau$-regular for all $t \ge 1$.
For each such $t$ there is a unique $\cZ_t \in \irr^\tau(A)$ such that
$M^t \in \cZ_t$.

\begin{Lem}\label{lem:limit1}
We have $I_{\cZ_{t+1}} \subseteq I_{\cZ_t}$.
\end{Lem}

\begin{proof}
Let $\bd_t := \dimv(M^t)$.
Let
$\cU$ be the image of the morphism
\begin{align*}
\GL_{\bd_{t+1}} \times \cZ_t \times \cO_M &\to \md(A,\bd_{t+1})
\\
(g,M_1,M_2) &\mapsto g.(M_1 \oplus M_2).
\end{align*}
Then $\cU$ is an irreducible constructible subset of
$\md(A,\bd_{t+1})$.
Furthermore, $\cU$ is $\tau$-regular, since $M^{t+1} \in \cU$.
Thus the $\tau$-regular modules form a dense subset of $\cU$.
It follows that $\cU \subseteq \cZ_{t+1}$.
Thus all modules from $\cZ_t$ appear as direct summands of modules
from $\cZ_{t+1}$.
This implies $I_{\cZ_{t+1}} \subseteq I_{\cZ_t}$.
\end{proof}

By Lemma~\ref{lem:limit1},
there is a chain
$$
\cdots \subseteq I_{\cZ_{t+1}} \subseteq I_{\cZ_t} \subseteq \cdots
\subseteq I_{\cZ_1}
$$
of ideals of $A$.
For dimension reasons there is some number $s(\cZ)$ such that
$$
I_{\cZ_t} = I_{\cZ_{s(\cZ)}}
$$ for all $t \ge s(\cZ)$.
Then
$$
I_{\cZ_\infty} := \bigcap_{t \ge 1} I_{\cZ_t} = I_{\cZ_{s(\cZ)}}.
$$
Define
$$
B := A/I_{\cZ_\infty}.
$$
The next lemma follows now immediately.

\begin{Lem}\label{lem:limit3}
With the notation above,
for all $t \ge s(\cZ)$, $\cZ_t$ is a faithful component in $\irr^\tau(B)$.
\end{Lem}

\begin{Thm}\label{thm:limit2}
Let $\cZ \in \irr^\tau(A)$.
Then the following hold:
\begin{itemize}\itemsep2mm

\item[(i)]
Assume that $I$ is an ideal of $A$ with $I\cZ = 0$, and let
$B = A/I$.
If
$$
\pdim_B(\cZ) \le 1,
$$
then $\cZ^t$ is $\tau_A$-regular for all
$t \ge 1$.

\item[(ii)]
Assume that $\cZ^t$ is $\tau_A$-regular for all $t \ge 1$,
let $I = I_{\cZ_\infty}$, and let $B = A/I_{\cZ_\infty}$.
Then $\pdim_B(\cZ) \le 1$.

\end{itemize}
\end{Thm}

\begin{proof}
(i):
Assume that $\pdim_B(\cZ) \le 1$.
It follows that $\cZ^t$ is $\tau_B$-regular for
all $t \ge 1$.
Let $M$ be generic in $\cZ$.
We get
\begin{align*}
c_A(M^t) &\ge c_B(M^t) = E_B(M^t) = t^2 E_B(M) = t^2 c_B(M)
  \\
&= t^2 c_A(M) = t^2 E_A(M) = E_A(M^t).
\end{align*}
This implies $c_A(M^t) = E_A(M^t)$.
Thus $M^t$ and therefore also $\cZ^t$ is $\tau_A$-regular
for all $t \ge 1$.

(ii):
Assume that $\cZ^t$ is $\tau_A$-regular for all $t \ge 1$.
Then $\cZ^t \subseteq \cZ_t$ for some unique
$\cZ_t \in \irr^\tau(A)$.

By Proposition~\ref{prop:reduction6},
each $\tau_A$-regular $M \in \cZ_t$ is also $\tau_B$-regular.
Thus $\cZ_t$ is $\tau_B$-regular for all $t \ge 1$.
Let $t \ge s(\cZ)$.
By Lemma~\ref{lem:limit3},
$\cZ_t$ is faithful (as a component of $B$-modules),
and $(\cZ_t)^s$ is generically $\tau_B$-regular
for all $s \ge 1$.
(Each $\tau_A$-regular module in $(\cZ^t)^s = \cZ^{ts}$ is $\tau_B$-regular
in $(\cZ_t)^s$, by Proposition~\ref{prop:reduction6}.)
Thus $\pdim_B(\cZ_t) \le 1$ by Theorem~\ref{thm:faithful4}.
It follows from
Theorem~\ref{thm:sincere}
that there is a $\tau_B$-regular
module $M = M_1 \oplus \cdots \oplus M_t$ in $\cZ_t$
with $M_i \in \cZ$ for all $1 \le i \le t$ and $\pdim_B(M) \le 1$.
This implies $\pdim_B(\cZ) \le 1$.
\end{proof}

\begin{Cor}\label{cor:limit3}
Let $\cZ \in \irr^\tau(A)$.
Then the following are equivalent:
\begin{itemize}\itemsep2mm

\item[(i)]
$\cZ^t$ is $\tau_A$-regular for all $t \ge 1$;

\item[(ii)]
There exists some ideal $I$ of $A$ with $I \cZ = 0$ such that
$$
\pdim_B(\cZ) \le 1,
$$
where $B = A/I$.

\end{itemize}
\end{Cor}

\begin{proof}
(i) $\implies$ (ii):
This follows directly from Theorem~\ref{thm:limit2}(ii).

(ii) $\implies$ (i):
This is the content of Theorem~\ref{thm:limit2}(i).
\end{proof}

\bigskip
\subsection*{Acknowledgements}
\quad
GB gratefully acknowledges the support of the National
Science Centre grant no.\ 2020/37/B/ST1/00127.
He also thanks the Mathematical Institute of the University of Bonn for two weeks of hospitality in October 2025.
JS was partially funded funded by the Deutsche Forschungsgemeinschaft (DFG, German Research Foundation) under Germany's Excellence Strategy – EXC-2047/2 – 390685813.


\end{document}